\documentclass[preprint, 12pt]{elsarticle}

\usepackage{graphicx}
\usepackage{subcaption}
\usepackage{amssymb, amsmath}
\usepackage{amsthm}
\usepackage{color}
\usepackage{array}
\usepackage{tabu}

\journal{Tr. Mat. Inst. Steklova}

\newtheorem{theorem}{Theorem}[section]
\newtheorem{lemma}[theorem]{Lemma}

\theoremstyle{definition}
\newtheorem{definition}[theorem]{Definition}

\theoremstyle{remark}
\newtheorem{remark}[theorem]{Remark}

\begin{document}

\begin{frontmatter}


\title{Some results on the existence of forced oscillations in mechanical systems}



\author{Ivan Polekhin}

\address{Steklov Mathematical Institute, Moscow, Russia}

\begin{abstract}
We present sufficient conditions for the existence of forced oscillations in non-autonomous mechanical systems. Previously, similar results were obtained for systems with friction. Presented results hold both for systems with and without friction. Some examples are given.
\end{abstract}
\end{frontmatter}


\section{Introduction}
\label{S:1}
Forced oscillations in systems of ordinary differential equations have been studied extremely intensively by many authors. Among all the works, a significant number of works are related to proving the existence of forced oscillations for the mathematical pendulum and pendulum-like systems, i.e. for systems that in some sense are similar to the pendulum. The first result on the matter was obtained by Georg Hamel \cite{hamel1922erzwungene}. He showed that equation
$$
\ddot x + g \cos x = c \sin t
$$
has a $2\pi$-periodic solution for any $c$. The proof is based on an application of the direct method of variational calculus and can be easily applied to the system where the right hand side is a relatively regular function with zero mean value. A detailed exposition of the papers of the second half of the 20th century related to the matter can be found in  \cite{mawhin1997seventy}.

In \cite{polekhin2014examples} (also see \cite{polekhin2014periodic}) it was shown that a $2\pi$-periodic solution also exists in the following system
$$
\ddot x + g \cos x + f(t) \sin t = 0,
$$
that describes the motion of a pendulum such that its pivot point is moving horizontally with acceleration $f(t)$ and this function is $2\pi$-periodic and regular. Moreover, it has been shown that for this solution we have $x(t) \in (0, \pi)$ for all $t$. In other words, for any relatively regular motion of the pivot, there exists a solution such that, along this solution, the rod of the pendulum is always above the horizontal line.

Similar result holds for the spherical pendulum provided the rod moves with friction (can be arbitrarily small). More general results have been proved in \cite{polekhin2015forced} and \cite{polekhin2016forced}. In particular, in \cite{polekhin2015forced} it has been shown the following. Let $N$ be a two-dimensional connected compact surface and there is a mass point moving on this surface. Let $M \subset N$ be a two-dimensional compact connected surface with a boundary. Surfaces $M$, $N$ and boundary $\partial M$ are supposed to be smooth manifolds embedded in $\mathbb{R}^3$. The equation of motion has the following general form:
$$
m \ddot r = F + F_{\textup{friction}} + F_{\textup{constraint}}.
$$
Here $m$ is the mass of the point, $r$ is the radius-vector, $F$ is an external force that can depend on $r$, $\dot r$ and time $t$, and this function is $T$-periodic in time, $F_{\textup{friction}}$ is the force of viscous friction with friction coefficient $\mu > 0$, $F_{\textup{constraint}}$ is the force of constraint directed along the normal vector to the surface. Let the Euler characteristic of $M$ is non-zero and all solutions starting at $\partial M$ and initially tangent to boundary $\partial M$ leave $M$ for some time interval $(0, \varepsilon]$. Then for any $\mu > 0$ there exists a $T$-periodic solution that remains in $M \setminus \partial M$ (provided some additional assumptions --- that we omit here --- are satisfied).

In \cite{bolotin2015calculus} Bolotin and Kozlov proved the following. Let us have a system with the following Lagrangian function
$$
L(q,\dot q, t) = L_2(q, \dot q, t) + L_1(q, \dot q, t) + L_0(q, t),
$$
where, as usual, $L_2$ is a positive definite quadratic form in velocity, $L_1$  is a linear in velocity form that defines gyroscopic forces. We assume that the Lagrangian function is $T$-periodic in time. Let manifold $N$ be the configuration space of the system. Let $M \subset N$ be a smooth manifold with boundary. Similar to the above, we assume that any solution starting at $\partial M$ ant tangent to the boundary
locally leaves $M$. Then in any homotopy class of free closed loops in $M$ there exists a $T$-periodic solution that always remains in $M \setminus \partial M$.

From \cite{bolotin2015calculus} it follows that there can be several  $T$-periodic solutions, However, this result cannot be applied to systems with friction. At the same time, results from \cite{polekhin2015forced} and \cite{polekhin2016forced} can be applied only to systems with friction (even though this friction can be chosen to be arbitrarily weak). The main aim of the paper is to prove results on the existence of forced oscillations for systems without friction that can be non-Hamiltonian.  

\section{Auxiliary results and definitions}

Everywhere below --- even if it is not mentioned explicitly --- we assume that all geometrical objects (functions, manifolds, vector fields) are infinitely smooth. We will use the result from~\cite{srzednicki2005fixed} (also, see \cite{srzednicki1994periodic}) based on the idea of the topological Wa{\.z}ewski method \cite{wazewski1947principe} and the Lefschetz--Hopf theorem. Let us now present the main result from~\cite{srzednicki2005fixed}, which we slightly modify for our use.

Let $v \colon \mathbb{R}\times M \to TM$ be a smooth non-autonomous vector field on a manifold $M$. In particular, the standard conditions for the existence and uniqueness for solutions are satisfied for equation
\begin{equation}
\label{eq1}
\dot x = v(t, x).
\end{equation}
For $t_0 \in \mathbb{R}$ and $x_0 \in M$ the mapping $t \mapsto x(t,t_0,x_0)$ is defined by the solution of the initial value problem for~(\ref{eq1}) with initial condition $x(0,t_0,x_0)=x_0$. For $W \subset \mathbb{R}\times M$ and $t\in\mathbb{R}$ we use the following notation 
\begin{equation*}
W_t=\{x \in M \colon (t,x) \in W\}.
\end{equation*}

\begin{definition}
Let $W \subset \mathbb{R} \times M$. Define the {exit set} $W^-$ as follows. A point $(t_0,x_0)$ is in $W^-$ if there exists $\delta>0$ such that $(t+t_0, x(t,t_0,x_0)) \notin W$ for all $t \in (0,\delta)$.
\end{definition}
\begin{definition}
We call $W \subset \mathbb{R}\times M$ a {Wa\.{z}ewski block} for the system $\dot x = v(t,x)$ if $W$ and $W^-$ are compact.
\end{definition}
 Now introduce some notations. By $\pi_1$ and $\pi_2$ we denote the projections of $\mathbb{R}\times M$ onto $\mathbb{R}$ and $M$ respectively. 
\begin{definition}
A set $W \subset [a,b] \times M$ is called a segment over $[a,b]$ if it is a block with respect to the system $\dot x = v(t,x)$ and the following conditions hold:
\begin{itemize}
\item there exists a compact subset $W^{--}$ of $W^-$ called the essential exit set such that
\begin{equation*}
W^-=W^{--}\cup(\{b\}\times W_b),\quad W^-\cap([a,b)\times M) \subset W^{--},
\end{equation*}
\item there exists a homeomorphism $h\colon [a,b]\times W_a \to W$ such that $\pi_1 \circ h = \pi_1$ and
\begin{equation}
\label{cond-2}
h([a,b]\times W_a^{--})=W^{--}.
\end{equation}
\end{itemize}
\end{definition}
\begin{definition}
Let $W$ be a segment over $[a,b]$. It is called periodic if
\begin{equation*}
(W_a,W_a^{--})=(W_b,W_b^{--}).
\end{equation*}
\end{definition}
\begin{definition}
For periodic segment $W$, we define the corresponding monodromy map $m$ as follows
\begin{equation*}
m\colon W_a\to W_a, \quad m(x) = \pi_2 h(b,\pi_2 h^{-1}(a,x)).
\end{equation*}
\end{definition}
\begin{remark}
Map $m$ is a homemorphism. Moreover, it can be shown that for various $h$ satisfying  (\ref{cond-2}), we have the same (up to homotopy) map $m$. Hence, the isomorphism of homology groups $m_*$ does not depend on the choice of $h$.
\end{remark}
\begin{theorem} 
\label{main-th} Let W be a periodic segment over $[a,b]$. Then the set
\begin{equation*}
U = \{ x_0 \in W_a \colon x(a,x_0,t) \in W_t\setminus W_t^{--}\,\mbox{for all}\,\, t \in [a,b] \}
\end{equation*}
is open in $W_a$ and the set of fixed point of the restriction $x(a,\cdot, b)|_U \colon U \to W_a$ is compact. Moreover, if $W$ and $W^{--}$ are ANRs \cite{dold2012lectures} then
\begin{equation*}
\mathrm{ind}(x(a,\cdot, b)|_U) = \Lambda(m) - \Lambda(m|_{W_a^{--}}).
\end{equation*}
Where by $\Lambda(m)$ and $\Lambda(m|_{W_a^{--}})$ we denote the Lefschetz number of $m$ and $m|_{W_a^{--}}$ respectively. In particular, if $\Lambda(m) - \Lambda(m|_{W_a^{--}}) \ne 0$ then $x(a,\cdot, b)|_U$ has a fixed point in $W_a$.
\end{theorem}

\begin{remark}
Let $W$ be as follows $W = [a, b] \times Z$, where $Z \subset M$, and $W^-_{t_1} = W^-_{t_2}$ for any $t_1, t_2 \in [a, b)$ and $W^{--}$ has the form 
$$
W^{--} = [a,b] \times W_a^-.
$$
Then $m$ is the identity map and the Lefschetz numbers equal to the Euler characteristics:
$$
\Lambda(m) - \Lambda(m|_{W_a^{--}}) = \chi(W_a) - \chi(W^-_a)= \chi(W_a) - \chi(W^{--}_a).
$$
\end{remark}
\begin{remark}
The requirement for $W$ and $W^{--}$ to be absolute neighborhood retracts significantly exceeds the needs of our applications. Below $W$ and $W^{--}$ are sets with piecewise smooth border.
\end{remark}
In conclusion of the section, let us prove a lemma from mechanics. Let again $M$ be an $n$-dimensional smooth manifold and $T$ is a positive definite quadratic form in velocity $\dot q$ defined on it:
\begin{align}
\label{quad-form}
    T = T(q, \dot q) = \frac{1}{2} A(q) \dot q \cdot \dot q.
\end{align}
This form defines a Riemannian metric $\langle \cdot, \cdot \rangle$ on $M$. Let us consider the following equation of motion
\begin{align}
\label{eq2}
    \nabla_{\dot q} \dot q = v(t, q, \dot q),
\end{align}
where, in local coordinates, $v(t, q, \dot q)$ are smooth functions of their arguments. When  $v \equiv 0$, we have the geodesic equation for metric $\langle \cdot, \cdot \rangle$
\begin{align}
\label{eq3}
    \nabla_{\dot q} \dot q = 0.
\end{align}
\begin{lemma}
\label{main-lm}
Let $q \colon U \to D$, $U \subset M$ be local coordinates in a region $U$ on $M$, $D \subset \mathbb{R}^n$ is a bounded open region in $\mathbb{R}^n$. We assume that in local coordinates all components of $v(t, q, \dot q)$ satisfy the following condition for all $t$, $q$ and $\dot q$
\begin{align}
\label{usl-rost}
    \| v(t, q, \dot q) \| \leqslant a + b\| \dot q \|^{2- \delta}
\end{align}
for some $\delta > 0$ and some $a$ and $b$. Let solution $q(t)$ of (\ref{eq3}) with initial conditions $q(0) = q_0 \in D$ and $\dot q(0) = \dot q_0 \ne 0$ be in $D$ for all $t$ in $[0,T]$. Then for any $\varepsilon > 0$ there exists $\lambda > 0$ such that for solutions $q_1(t)$ and $q_2(t)$ of (\ref{eq2}) and (\ref{eq3}) with initial conditions $q_1(0) = q_2(0) = q_0$, $\dot q_1(0) = \dot q_2(0) = \lambda \dot q_0$ we have
$$
\| q_1(t) - q_2(t) \| < \varepsilon, \quad t \in [0, T/\lambda].
$$
\end{lemma}
\begin{proof}
Let us consider a new time variable $t = \tau/\lambda$. Equation (\ref{eq2}) has the form
\begin{align}
\label{eq4}
    \nabla_{q'} q' = \frac{1}{\lambda^2}v(\tau/\lambda, q, \lambda q'),
\end{align}
where $(\cdot)'$ is the derivative by $\tau$. Equation (\ref{eq3}) does not change:
\begin{align}
\label{eq5}
    \nabla_{q'} q' = 0.
\end{align}
Let us denote by $Q_1(\tau)$ and $Q_2(\tau)$, the corresponding solutions of (\ref{eq4}) and (\ref{eq5}) with initial conditions $Q_1(0) =  Q_2(0) = q_0$, $ Q'_1(0) =  Q'_2(0) = \dot q_0$.
From the statement of the lemma, we have that $Q_2(\tau)$ in $D$ for $\tau \in [0,  T]$. If  $\lambda$ is large enough, the solutions of (\ref{eq4}) and (\ref{eq5}) with the above initial conditions are arbitrarily close for $t$ in $[0, T]$. This follows from (\ref{usl-rost}). In particular,
$$
\| Q_1(\tau) -  Q_2(\tau) \| < \varepsilon, \quad \tau \in [0, T].
$$
Since $q_1(t) = Q_1(\lambda t)$ and $q_2(t) = Q_2(\lambda t)$, then
$$
\| q_1( t) - q_2( t) \| = \| Q_1(\lambda t) - Q_2(\lambda t) \|  < \varepsilon, \quad t \in [0, T/\lambda].
$$
\end{proof}

Therefore, we have shown that for relatively large initial velocity, the trajectory of solution for (\ref{eq4}) is close to the corresponding geodesic (for some time interval). Inequality (\ref{usl-rost}) is important to the proof.

\begin{remark}
Below, we will use the following observation: if for $\lambda$ the corresponding solutions of equations (\ref{eq2}) and (\ref{eq3}) are close, then the same is true if we change the right hand side of (\ref{eq2}) as follows
\begin{align}
\label{eq2-new}
    \nabla_{\dot q} \dot q = w(t, q, \dot q),
\end{align}
provided for all $t$, $q$ and $\dot q$ we have
\begin{align}
    \| w(t, q, \dot q) \| \leqslant \| v(t, q, \dot q) \|.
\end{align}
This follows from the continuous dependence of solutions from the right hand side of the equation.
\end{remark}

\section{Main results}

In \cite{polekhin2014examples} it was shown that, for the system describing the motion of a planar mathematical pendulum with a periodically horizontally moving pivot point, there is a periodic solution with the same period such that the rod of the pendulum is never horizontal along this solution. To be more precise, let us have te following system
\begin{align}
\label{eq6}
    \ddot q = f(t) \cdot \sin q - \cos q.
\end{align}
Where $f(t)$ is $T$-periodic. Then there exists a solution $q(t)$ such that
\begin{enumerate}
    \item $q(t) \in (0, \pi)$ for all $t$,
    \item $q(t + T) = q(t)$ for all $t$.
\end{enumerate}
The proof is based on a result from \cite{srzednicki2005fixed} presented above in a simplified form. In the proof, we use that it is possible to choose a periodic segment for system (\ref{eq6}). Qualitative picture of the segment is in Fig.~1.
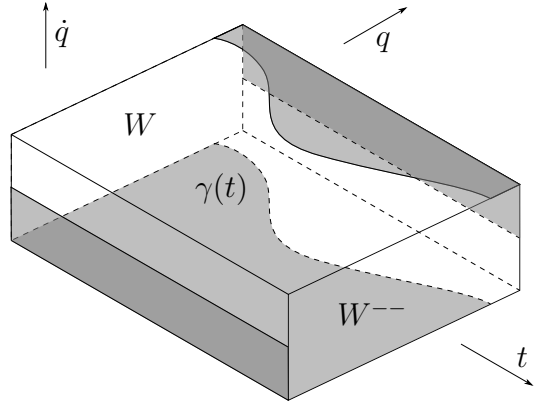
\begin{figure}[h!]
\centering
\def\svgwidth{200 pt}
\begingroup%
  \makeatletter%
  \providecommand\color[2][]{%
    \errmessage{(Inkscape) Color is used for the text in Inkscape, but the package 'color.sty' is not loaded}%
    \renewcommand\color[2][]{}%
  }%
  \providecommand\transparent[1]{%
    \errmessage{(Inkscape) Transparency is used (non-zero) for the text in Inkscape, but the package 'transparent.sty' is not loaded}%
    \renewcommand\transparent[1]{}%
  }%
  \providecommand\rotatebox[2]{#2}%
  \newcommand*\fsize{\dimexpr\f@size pt\relax}%
  \newcommand*\lineheight[1]{\fontsize{\fsize}{#1\fsize}\selectfont}%
  \ifx\svgwidth\undefined%
    \setlength{\unitlength}{451.61904176bp}%
    \ifx\svgscale\undefined%
      \relax%
    \else%
      \setlength{\unitlength}{\unitlength * \real{\svgscale}}%
    \fi%
  \else%
    \setlength{\unitlength}{\svgwidth}%
  \fi%
  \global\let\svgwidth\undefined%
  \global\let\svgscale\undefined%
  \makeatother%
  \begin{picture}(1,0.75404713)%
    \lineheight{1}%
    \setlength\tabcolsep{0pt}%
    \put(0,0){\includegraphics[width=\unitlength,page=1]{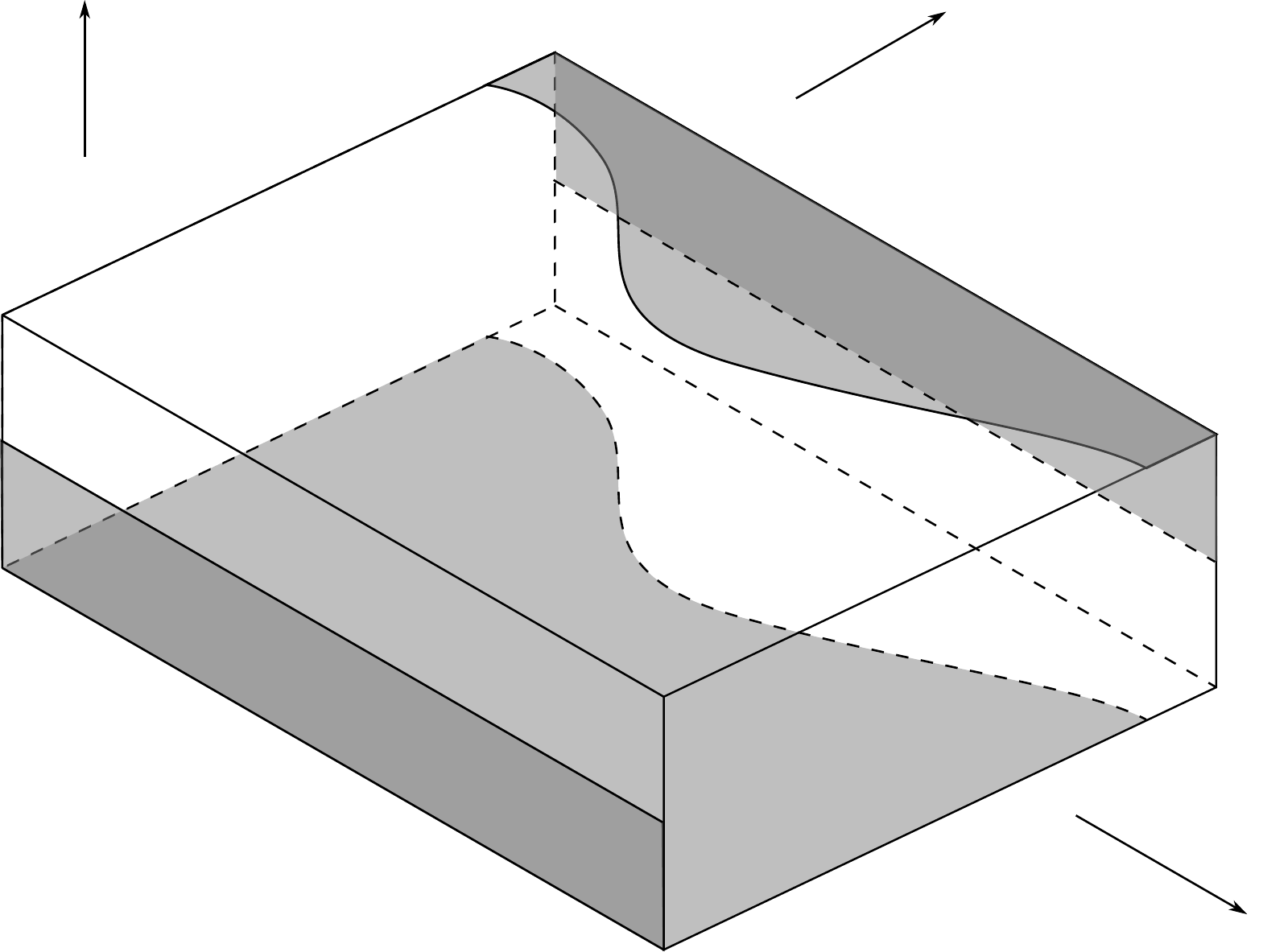}}%
    \put(0.3494258,0.38660335){\color[rgb]{0,0,0}\makebox(0,0)[lt]{\lineheight{1.25}\smash{\begin{tabular}[t]{l}$\gamma(t)$\end{tabular}}}}%
    \put(0.69105387,0.67247953){\color[rgb]{0,0,0}\makebox(0,0)[lt]{\lineheight{1.25}\smash{\begin{tabular}[t]{l}$q$\end{tabular}}}}%
    \put(0.08371515,0.68078299){\color[rgb]{0,0,0}\makebox(0,0)[lt]{\lineheight{1.25}\smash{\begin{tabular}[t]{l}$\dot q$\end{tabular}}}}%
    \put(0.95795077,0.06514083){\color[rgb]{0,0,0}\makebox(0,0)[lt]{\lineheight{1.25}\smash{\begin{tabular}[t]{l}$t$\end{tabular}}}}%
    \put(0.21419806,0.50403797){\color[rgb]{0,0,0}\makebox(0,0)[lt]{\lineheight{1.25}\smash{\begin{tabular}[t]{l}$W$\end{tabular}}}}%
    \put(0.61869514,0.14343063){\color[rgb]{0,0,0}\makebox(0,0)[lt]{\lineheight{1.25}\smash{\begin{tabular}[t]{l}$W^{--}$\end{tabular}}}}%
  \end{picture}%
\endgroup%

\caption{Periodic segment for the pendulum with a moving pivot point. The exit set is highlighted ($W^{--}$).}
\label{fig1}
\end{figure}
In other words, there exists a parallelepiped in the extended phase space such that $W^{--}$ is compact and has the following form:
 $$W^{--} = W^{--}_{left} \cup W^{--}_{right} \cup W^{--}_{top} \cup W^{--}_{bottom},$$
where
\begin{align*}
    \begin{split}
        &W^{--}_{right} = \{ t, q, \dot q \colon t \in [0, T], q = 0, \dot q \in [-p, 0] \},\\
        &W^{--}_{left} = \{ t, q, \dot q \colon t \in [0, T], q = \pi, \dot q \in [0, p] \},\\
        &W^{--}_{top} = \{ t, q, \dot q \colon t \in [0, T], q \in [\gamma(t),\pi], \dot q = p \},\\
        &W^{--}_{bottom} = \{ t, q, \dot q \colon t \in [0, T], q \in [0,\gamma(t)], \dot q = -p \}.
    \end{split}
\end{align*}
Here $p > 0$ is a relatively large number, for given $t$, the value $\gamma(t) \in (0, \pi)$ satisfies condition $\mathrm{ctg} \gamma(t) = f(t)$.

If one considers the motion of the point not along a circle, but on an arbitrary smooth curve, then the equations of motion have the form:
\begin{align*}
    \ddot q = f(t) \cdot x'(q) - y'(q),
\end{align*}
where $(x(q), y(q))$ is a natural parametrization of the curve. The point is moving in a gravitational field and in a horizontal non-autonomous external force field. For this problem, set $$W = \{t, q, \dot q \colon t \in [0,T], q \in [0, \pi], \dot q \in [-p, p]\}$$ may not be a periodic segment: solutions can be internally tangent to the boundary of $W$. Nevertheless, $W$ is a periodic segment for the system with friction:
\begin{align*}
    \ddot q = f(t) \cdot x'(q) - y'(q) - \mu \dot q,
\end{align*}
where $\mu > 0$ is an arbitrary real number. Below we show that this system has a $T$-periodic solution even when $\mu = 0$.
\begin{remark}
Note that the existence of a periodic solution for $\mu = 0$ cannot be obtained by a passage to the limit, since the size of $W$ is unbounded as $\mu$ tends to zero.
\end{remark}
\noindent Let us consider the following equation
\begin{align}
\label{eq7}
    \ddot x = v(t, x, \dot x),
\end{align}
where $v \colon \mathbb{R} \times \mathbb{R} \times \mathbb{R} \to \mathbb{R}$ is a smooth function.
\begin{theorem}
\label{3_2}
Let the right hand side of (\ref{eq7}) be a $T$-periodic function of time and there exist two $T$-periodic functions $x_1(t)$ and $x_2(t)$ such that for all $t$ we have
\begin{enumerate}
    \item $x_1(t) < x_2(t)$,
    \item $\ddot x_1 > v(t, x_1(t), \dot x_1(t))$,
    \item $\ddot x_2 < v(t, x_2(t), \dot x_2(t))$.
\end{enumerate}
Assume that $v(t, x, \dot x) \leqslant a + b|\dot x|^{2 - \delta}$ for some $a$ and $b$ and $\delta > 0$. Then for system (\ref{eq7}) there exists a periodic solution $x(t)$ such that $x_1(t) < x(t) < x_2(t)$ for all $t$.
\end{theorem}
\begin{proof}
Let us introduce the following notations $x^- = \min x_1(t)$ and $x^+ = \max x_2(t)$. Let $p > 0$ and $\varepsilon > 0$ be such numbers that any solution of (\ref{eq7}) with initial conditions $x(t_0) \in [x_1(t_0), x_2(t_0)]$ and $\dot x(t_0) \in [-\varepsilon + p, \varepsilon + p]$ or $\dot x(t_0) \in [-\varepsilon - p, -\varepsilon + p]$ for any $t_0 \in [0,T]$ leaves segment $[x^--1, x^++1]$. The existence of such $p$ and $\varepsilon$ follows from Lemma \ref{main-lm} for the standard Euclidean metric on the line. Let us now consider the equation
\begin{align}
\label{eq8-temp}
    \ddot x = v(t, x, \dot x)\chi_{p,\varepsilon}(\dot x),
\end{align}
where $\chi_{p,\varepsilon}(\dot x)$ is a smooth function such that
$$
\chi_{p,\varepsilon}(\dot x) = 
\begin{cases}
1, \quad \dot x \in (-\infty, -p-\varepsilon] \cup [-p+\varepsilon, p-\varepsilon] \cup [p + \varepsilon, \infty),\\
0, \quad \dot x \in [-p-\frac{\varepsilon}{2}, -p+\frac{\varepsilon}{2}] \cup [p-\frac{\varepsilon}{2}, p+\frac{\varepsilon}{2}],\\
\mbox{monotone in all other intervals.}
\end{cases}
$$
In other words, equation (\ref{eq8-temp}) everywhere coincides with the original one except for some subset of the extended phase space for which, in particular, we make the right hand side of (\ref{eq7}) equal to zero. For (\ref{eq8-temp}) we have
$$
\| v(t, x, \dot x)\chi_{p,\varepsilon}(\dot x) \| \leqslant \| v(t, x, \dot x) \|.
$$
Hence, every solution of (\ref{eq8-temp}) with initial conditions $x(t_0) \in [x_1(t_0), x_2(t_0)]$ and $\dot x(t_0) \in [-\varepsilon + p, \varepsilon + p]$ or $\dot x(t_0) \in [-\varepsilon - p, -\varepsilon + p]$ for any $t_0 \in [0,T]$ also leaves $[x^-, x^+]$.

Let us consider another modified equation:
\begin{align}
\label{eq8-temp2}
    \ddot x = v(t, x, \dot x)\chi_{p,\varepsilon}(\dot x) - \mu\dot x (1 - \chi_{p,\varepsilon}(\dot x)),
\end{align}
where $\mu > 0$ is a small number such that any solution of (\ref{eq8-temp2}) with initial conditions $x(t_0) \in [x_1(t_0), x_2(t_0)]$ and $\dot x(t_0) \in [-\varepsilon + p, \varepsilon + p]$ or $\dot x(t_0) \in [-\varepsilon - p, -\varepsilon + p]$ leaves $[x^-, x^+]$. The existence of such $\mu$ follows from the continuous dependence of solutions from the right hand side of the equation.

\begin{figure}[h!]
\centering
\def\svgwidth{200 pt}
\begingroup%
  \makeatletter%
  \providecommand\color[2][]{%
    \errmessage{(Inkscape) Color is used for the text in Inkscape, but the package 'color.sty' is not loaded}%
    \renewcommand\color[2][]{}%
  }%
  \providecommand\transparent[1]{%
    \errmessage{(Inkscape) Transparency is used (non-zero) for the text in Inkscape, but the package 'transparent.sty' is not loaded}%
    \renewcommand\transparent[1]{}%
  }%
  \providecommand\rotatebox[2]{#2}%
  \newcommand*\fsize{\dimexpr\f@size pt\relax}%
  \newcommand*\lineheight[1]{\fontsize{\fsize}{#1\fsize}\selectfont}%
  \ifx\svgwidth\undefined%
    \setlength{\unitlength}{500.48437214bp}%
    \ifx\svgscale\undefined%
      \relax%
    \else%
      \setlength{\unitlength}{\unitlength * \real{\svgscale}}%
    \fi%
  \else%
    \setlength{\unitlength}{\svgwidth}%
  \fi%
  \global\let\svgwidth\undefined%
  \global\let\svgscale\undefined%
  \makeatother%
  \begin{picture}(1,0.68024798)%
    \lineheight{1}%
    \setlength\tabcolsep{0pt}%
    \put(0,0){\includegraphics[width=\unitlength,page=1]{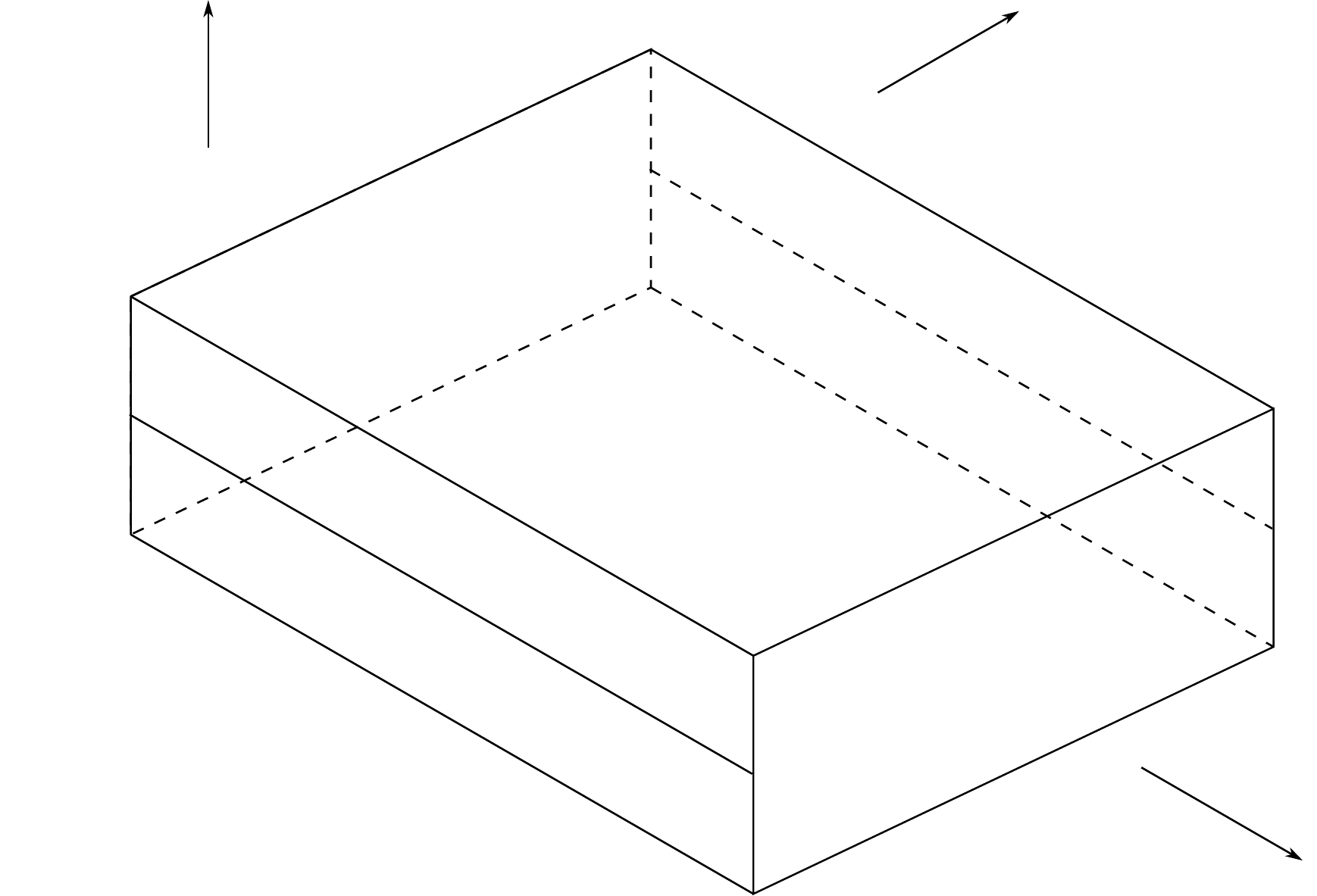}}%
    \put(0.7212182,0.60664431){\color[rgb]{0,0,0}\makebox(0,0)[lt]{\lineheight{1.25}\smash{\begin{tabular}[t]{l}$q$\end{tabular}}}}%
    \put(0.17317765,0.61413705){\color[rgb]{0,0,0}\makebox(0,0)[lt]{\lineheight{1.25}\smash{\begin{tabular}[t]{l}$\dot q$\end{tabular}}}}%
    \put(0.96205629,0.0586038){\color[rgb]{0,0,0}\makebox(0,0)[lt]{\lineheight{1.25}\smash{\begin{tabular}[t]{l}$t$\end{tabular}}}}%
    \put(0,0){\includegraphics[width=\unitlength,page=2]{2019-10-W-II.pdf}}%
    \put(0.04319103,0.42026493){\color[rgb]{0,0,0}\makebox(0,0)[lt]{\lineheight{1.25}\smash{\begin{tabular}[t]{l}$p$\end{tabular}}}}%
    \put(-0.00199026,0.28649418){\color[rgb]{0,0,0}\makebox(0,0)[lt]{\lineheight{1.25}\smash{\begin{tabular}[t]{l}$-p$\end{tabular}}}}%
  \end{picture}%
\endgroup%

\caption{Solutions with large $|\dot q(0)|$ leave $[x^-, x^+]$.}
\label{fig1}
\end{figure}

Let us consider a periodic segment $W$ for modified equation (\ref{eq8-temp2}):
\begin{align}
    W = \{ t, x, \dot x \colon t \in [0,T], x \in [x_1(t), x_2(t)], \dot x \in [-p, p] \}.
\end{align}
From \ref{main-th} it follows that there exists a $T$-periodic solution of (\ref{eq8-temp2}) that never leaves $W$. The proof is similar to the proofs in \cite{polekhin2014examples,polekhin2014periodic,polekhin2015forced}. This solution cannot be in points where $\sigma_{p, \varepsilon} \neq 0$. Indeed, in this case, the solution leaves $[x^-,x^+]$ and, hence, leaves $W$. From the contradiction we obtain that this $T$-periodic solution of (\ref{eq8-temp2}) always remains in the region where $\sigma_{p, \varepsilon} = 0$, i.e. this solution is a solution for the original equation (\ref{eq7}).
\end{proof}

The main idea of the proof can be explained as follows. When one considers the system with friction, it is possible to use results from \cite{srzednicki1994periodic}. However, the existence of a periodic solutions follows from the behaviour of the vector field in a vicinity of the boundary of segment $W$ and does not depend on the fact whether there is friction inside segment $W$ or not. Therefore, it is sufficient to introduce the friction only for relatively large (in absolute values) velocities. Then it is possible to show that any solution with large velocity leaves our segment $W$ and cannot be our $T$-periodic solution. Finally, our $T$-periodic solution always stays in the region without friction.

This result can be generalized for the case of system of $n$ differential equations of second order:
\begin{align}
\label{eq9}
\begin{split}
    &\ddot x_1 = v_1(t, x_1, \dot x_1, x_2, \dot x_2, ..., x_n, \dot x_n),\\
    &\ddot x_2 = v_2(t, x_1, \dot x_1, x_2, \dot x_2, ..., x_n, \dot x_n),\\
    &...\\
    &\ddot x_n = v_n(t, x_1, \dot x_1, x_2, \dot x_2, ..., x_n, \dot x_n).
\end{split}
\end{align}
\begin{theorem}
Let the right hand side of (\ref{eq9}) is a $T$-periodic function of time and there are $2n$ $T$-periodic functions $x_1^n(t)$ and $x_2^n(t)$ such that for all $t$ 
\begin{enumerate}
    \item $x_1^j(t) < x_2^j(t)$, $1 \leqslant j \leqslant n$
    \item $\ddot x_1^j(t) > v_j(t, x_1(t), \dot x_1, x_2(t), \dot x_2, ..., x^j_1(t), \dot x^j_1(t), ..., x_n(t), \dot x_n)$ for any $\dot x_i$ and any $x_i(t) \in [x_1^i(t), x_2^i(t)]$,
    \item $\ddot x_2^j(t) < v_j(t, x_1(t), \dot x_1, x_2(t), \dot x_2, ..., x^j_2(t), \dot x^j_2(t), ..., x_n(t), \dot x_n)$  for any $\dot x_i$ and any $x_i(t) \in [x_1^i(t), x_2^i(t)]$.
\end{enumerate}
Let $\|v(t, x, \dot x)\| \leqslant a + b\|\dot x\|^{2 - \delta}$ for some $a$, $b$ and $\delta > 0$. Then for system (\ref{eq7}) there exists a $T$-periodic solution $$x(t) = (x_1(t), x_2(t), ..., x_n(t))$$ such that $x_1^j(t) < x_j(t) < x_2^j(t)$ for all $t$.
\end{theorem}

Similar result can be proved for a system with two degrees of freedom.
\begin{theorem}
\label{th-sphere}
Let us have a simply connected open region $D \subset \mathbb{R}^2$ with a smooth boundary. Let $L_2$ be a quadratic form (\ref{quad-form}) on $\mathbb{R}^2$. Let the trajectory of any solution of
\begin{align}
    \nabla_{\dot q}\dot q = 0
\end{align}
such that $\langle \dot q(0), \dot q(0) \rangle = 1$ and $q(0) \in D$ leave some open neighborhood of $D$ in time not more than $\tau$, one for all initial conditions. Let us consider the equation
\begin{align}
\label{eq18}
    \nabla_{\dot q}\dot q = v(t, q, \dot q),
\end{align}
where $v(t, q, \dot q)$ is a $T$-periodic in time function such that $$\|v(t, x, \dot x)\| \leqslant a + b\|\dot x\|^{2 - \delta}$$ for some $a$, $b$ and $\delta > 0$. Let all solution, that are tangent to boundary $\partial D$ at time $t_0$ (for any $t_0$), locally leave $D \cup \partial D$. Then there is a $T$-periodic solution in $D$.
\end{theorem}
\begin{proof}
The general scheme of the proof is the same as for Theorem~\ref{3_2} and here we present the proof without details. Let $p$ and $\varepsilon$ be real numbers such that any solution $q(t)$ of (\ref{eq18}) satisying 
$$
q(t_0) = q_0 \in D \cup \partial D, \quad \dot q(t_0) = \dot q_0, \quad 2(p-\varepsilon) \leqslant \langle \dot q_0, \dot q_0 \rangle \leqslant 2(p + \varepsilon)
$$
leaves some open neighborhood of $D$. Let us denote the set of such  $q_0$ and $\dot q_0$ by $X$. Now we can consider the modified equation that differs from the original one only in $X$:
\begin{align}
\label{eq19}
    \nabla_{\dot q}\dot q = \tilde v(t, q, \dot q),
\end{align}
We assume that smooth function $\tilde v(t, q, \dot q)$ satisfies the following condition: $\nabla_{\dot q} \langle \dot q, \dot q \rangle < 0$ for $\langle \dot q, \dot q \rangle = 2p$. The modified equation can always be chosen in such a way that any its solution, with initial conditions in $X$, leaves a neighborhood of $D$ in finite time. Let us consider the set
\begin{align}
    W = \{ t, q, \dot q \colon t \in [0, T], q \in D \cup \partial D, \langle \dot q, \dot q \rangle = 2p \}.
\end{align}
For equation (\ref{eq19}) set $W$ is a periodic segment. The exit set is homotopic to a circle. Hence, for the modified equation, there is a $T$-periodic solution that always remains in $W$. This solution cannot go through set $X$ since all solutions starting in $X$ leave $W$. Finally, we obtain that there exists a periodic solution for (\ref{eq18}).
\end{proof} 

Theorem (\ref{th-sphere}) can be generalized to the case of an $n$-dimensional region. The proof is almost the same as for the previous result. The only difference is that the exit set is homotopic to an $(n-1)$-sphere. Its Euler characteristic is $0$ ($n$ is even) or $2$ ($n$ is odd).
\begin{theorem}
\label{th-n}
Let us have a simply connected open region $D \subset \mathbb{R}^n$ with a smooth boundary. Let $L_2$ be a quadratic form (\ref{quad-form}) on $\mathbb{R}^n$. Let the trajectory of any solution of
\begin{align}
    \nabla_{\dot q}\dot q = 0
\end{align}
such that $\langle \dot q(0), \dot q(0) \rangle = 1$ and $q(0) \in D$ leave some open neighborhood of $D$ in time not more than $\tau$, one for all initial conditions. Let us consider the equation
\begin{align}
\label{eq18}
    \nabla_{\dot q}\dot q = v(t, q, \dot q),
\end{align}
where $v(t, q, \dot q)$ is a $T$-periodic in time function such that $$\|v(t, x, \dot x)\| \leqslant a + b\|\dot x\|^{2 - \delta}$$ for some $a$, $b$ and $\delta > 0$. Let all solution, that are tangent to boundary $\partial D$ at time $t_0$ (for any $t_0$), locally leave $D \cup \partial D$. Then there is a $T$-periodic solution in $D$.
\end{theorem}

\section{Examples}
\subsection{Pendulum with a feedback control}
We have already considered the equation of motion for a pendulum with a moving pivot point (\ref{eq6}). Let us now consider a more general equation
\begin{align}
    \ddot q = f(t,q,\dot q) \cdot \sin q - \cos q.
\end{align}
This equation describes the dynamics of a pendulum in an external horizontal force field $f$. The same equation describes the motion of a feedback controlled pendulum when the feedback control is realized by a horizontal motion of the pivot (we suppose that the mass of the pendulum is much less than the mass of the support cart \cite{polekhin2018topological}).

Let $f(t,q,\dot q)$ be a $T$-periodic bounded function. The boundedness of $f$ is natural for control problems since our feedback forces, in practice, cannot be infinite. Note that, in the general case, our equations cannot be rewritten as Lagrange equations, i.e. the dynamics is not defined by the Hamilton's variational principle.

From Theorem~\ref{3_2}, we have that for this system there always exists a $T$-periodic solution. Indeed, one can consider functions $s_1(t) = 0$ and $s_2(t) = \pi$ and they satisfy the conditions of the theorem. Moreover, it is easy to see that there is a periodic solution along which the pendulum never falls.

\subsection{Spherical pendulum}
Similar result can be proved for a spherical pendulum. Let us have a a force of gravity and a horizontal $T$-periodic function of time $F(t, r, \dot r) = F_x e_x + F_y e_y$ acting on the pendulum. Here $e_x$ and $e_y$ are unit vectors in the horizontal plane. We assume that $F_x(t, r, \dot r)$ and $F_y(t, r, \dot r)$ are bounded. If the mass of the pendulum and the gravity acceleration equal $1$, then the equation of motion has the form
\begin{align}
    \ddot r = -e_{z} + F + F_{\textup{constraint}},
\end{align}
where $F_{constraint}$ is the force of constraint directed along the normal vector, $e_z$ is the vertical vector, $r \in \mathbb{R}^3$. We also assume that the length of the rod equals $1$. Let us consider the following subset of the configuration space:
\begin{align}
    G = \{ r \colon (r, e_z) > \varepsilon \} \quad\mbox{for some}\quad \varepsilon > 0.
\end{align}
Since $F$ is bounded, then for small $\varepsilon$ we have that the corresponding solutions are externally tangent to $G$. The geodesics are  the arcs of large circles, therefore, any geodesic starting in $G$ with unit velocity reaches the horizontal large circle (equator) in time no more than $\pi$. We can apply Theorem~\ref{th-sphere} and obtain that there exists a $T$-periodic solution that always remains in $G$.

\subsection{Mass point on a rotating curve}
Let us consider the following mechanical system. Let us have a closed smooth strictly convex curve $\gamma$ in a vertical plane. Let this curve rotate around in the vertical plane around some fixed point (Fig.~\ref{fig1}). We assume that the law of rotation $\varphi$ is a $T$-periodic function. We also assume that there is a unit mass point sliding along the curve without any friction.
\begin{figure}[h!]
\centering
\def\svgwidth{250 pt}
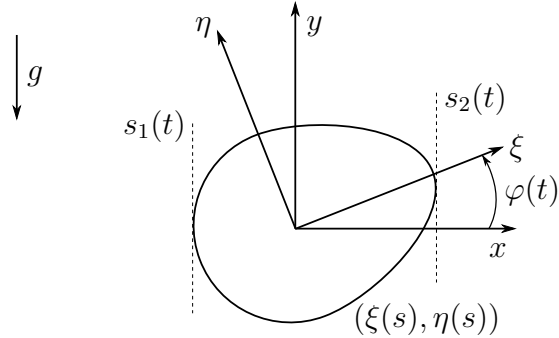
\caption{Smooth curve in a vertical plane.}
\label{fig1}
\end{figure}
Let $s$ be a natural parameter on curve $\gamma$ that is given by its components $\xi(s)$ and $\eta(s)$ in Cartesian coordinate system $O\xi\eta$. The kinetic energy has the form:
\begin{align}
    T = \frac{1}{2}\left( \dot s^2 + \dot\varphi^2(\xi^2 + \eta^2) + 2\dot\varphi\dot s (\xi\eta' - \eta\xi') \right).
\end{align}
If the gravity acceleration equals $1$, then the potential of system has the form:
\begin{align}
    V = - (\xi\sin\varphi + \eta \cos\varphi).
\end{align}
The equation of motion:
\begin{align}
\label{ex_3}
    \ddot s + \ddot\varphi(\xi\eta' - \eta\xi') - \dot\varphi^2(\xi\xi' + \eta\eta') +  (\xi'\sin\varphi + \eta' \cos\varphi) = 0.
\end{align}
For a given $t$, there are two points $s_1(t)$ and $s_2(t)$ on the curve such that
\begin{align}
    \begin{split}
        & \xi' (s_1(t)) \sin\varphi(t) + \eta'(s_1(t)) \cos\varphi(t) = -1,\\
        & \xi' (s_2(t)) \sin\varphi(t) + \eta'(s_2(t)) \cos\varphi(t) = 1.
    \end{split}
\end{align}
We assume that the orientation of the curve is chosen in such a way that
 $$s_1(t) < s_2(t).$$ Functions $s_1(t)$ and $s_2(t)$ are smooth. Since $\xi(s)$, $\eta(s)$, $\xi'(s)$, $\eta'(s)$ are bounded, then if $|\dot\varphi(t)|$ and $|\ddot\varphi(t)|$ are small for all $t$, then we can apply Theorem~\ref{3_2}. We obtain that, for system (\ref{ex_3}), there exists a $T$-periodic solution that always remains between $s_1(t)$ and $s_2(t)$. In other words, for relatively slow and smooth rotation, there exists a $T$-periodic solution such that along this solution the mass point is always in `the top part of the curve'.

Similar result can be proved without the assumption of convexity of the curve. For instance, one can assume that function $\varphi(t)$ satisfies some additional assumptions guaranteeing that functions $s_1(t)$ and $s_2(t)$ are smooth.

For instance, let us consider the following curve (Fig. \ref{fig2}). If one rotates this curve uniformly, then $s_1(t)$ and $s_2(t)$ becomes discontinuous. However, if we rotates it in a special manner, the functions can be continuous. For instance, we can define $\varphi(t)$ as the angle between the vertical and the tangent vector when we uniformly moves along the curve (Fig. \ref{fig2}). Then, for small velocity and acceleration of rotation, there exists a $T$-periodic solution with the same properties as above.

\begin{figure}[h!]
\centering
\def\svgwidth{160 pt}
\begingroup%
  \makeatletter%
  \providecommand\color[2][]{%
    \errmessage{(Inkscape) Color is used for the text in Inkscape, but the package 'color.sty' is not loaded}%
    \renewcommand\color[2][]{}%
  }%
  \providecommand\transparent[1]{%
    \errmessage{(Inkscape) Transparency is used (non-zero) for the text in Inkscape, but the package 'transparent.sty' is not loaded}%
    \renewcommand\transparent[1]{}%
  }%
  \providecommand\rotatebox[2]{#2}%
  \newcommand*\fsize{\dimexpr\f@size pt\relax}%
  \newcommand*\lineheight[1]{\fontsize{\fsize}{#1\fsize}\selectfont}%
  \ifx\svgwidth\undefined%
    \setlength{\unitlength}{185.09487279bp}%
    \ifx\svgscale\undefined%
      \relax%
    \else%
      \setlength{\unitlength}{\unitlength * \real{\svgscale}}%
    \fi%
  \else%
    \setlength{\unitlength}{\svgwidth}%
  \fi%
  \global\let\svgwidth\undefined%
  \global\let\svgscale\undefined%
  \makeatother%
  \begin{picture}(1,0.63410209)%
    \lineheight{1}%
    \setlength\tabcolsep{0pt}%
    \put(0,0){\includegraphics[width=\unitlength,page=1]{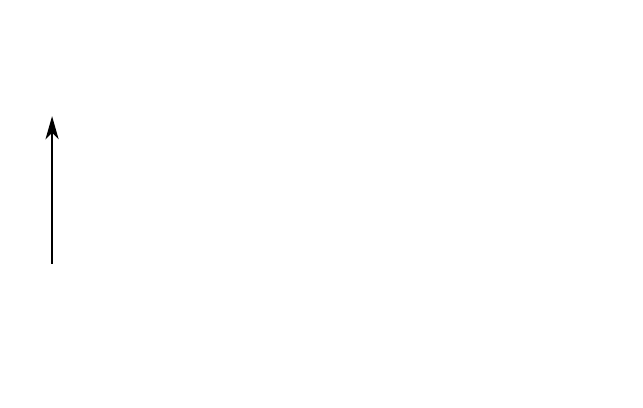}}%
    \put(-0.00358769,0.33052379){\color[rgb]{0,0,0}\makebox(0,0)[lt]{\lineheight{1.25}\smash{\begin{tabular}[t]{l}$v_0$\end{tabular}}}}%
    \put(0.31767607,0.49260283){\color[rgb]{0,0,0}\makebox(0,0)[lt]{\lineheight{1.25}\smash{\begin{tabular}[t]{l}$v_t$\end{tabular}}}}%
    \put(0.16717418,0.57364236){\color[rgb]{0,0,0}\makebox(0,0)[lt]{\lineheight{1.25}\smash{\begin{tabular}[t]{l}$\varphi(t)$\end{tabular}}}}%
    \put(0,0){\includegraphics[width=\unitlength,page=2]{2019-10-curve-2.pdf}}%
  \end{picture}%
\endgroup%

\caption{Non-convex curve in a vertical plane.}
\label{fig2}
\end{figure}

\subsection{Nonlinear chain in an external field}
Let us have a finite chain of particles on a line. Let there be $n+2$ particles and the leftmost and rightmost particles are fixed in points $x = 0$ and $x = 2(n+1)$. The potential of interaction has the form
$$
V(x) = \frac{1}{2}\left( 1 - e^{-(x-1)} \right)^2.
$$
This potential is called the Morse potential: when the distance between two particles is less than $1$, they repulse. Otherwise, they attract each other. Starting from some distance, the force of interaction is decreasing monotonously.  

Let us now have the following equation
$$
\ddot x_i = -\frac{\partial V}{\partial x}\Big|_{x = x_{i+1} - x_i} + \frac{\partial V}{\partial x}\Big|_{x = x_{i} - x_{i-1}} + F(t, x_i), \quad i = 1,...,n.
$$
Where $F(t, x)$ is an external field in point $x$ at time $t$. We assume that $F(t+T,x) = F(t,x)$ and:
\begin{itemize}
    \item $F(t,2(k+1)) < 0$ for all $t$,
    \item $F(t,2k) > 0$ for all $t$.
\end{itemize}
Then there exists a $T$-periodic solution such that $$x_i(t) \in (2(2i - 1), 4i)$$ fora all $i$ and all $t$. It follows from Theorem 3.3 and the observation that for $x_i = 2(2i-1)$ we always have $\ddot x_i < 0$, since the $(i-1)$-th point attracts the $i$-th point not weaker than the $(i+1)$-th point. Similarly, for $x_i = 4i$ we have $\ddot x_i > 0$.

\section{Conclusion}
In conclusion we would like to say something about the connection of the above results with the previously obtained in the literature and to comment our examples.

Theorem~\ref{3_2} is a special case of a more general result on the existence of a solution for the following boundary value problem
$$
u'' = f(t, u, u'), \quad u(a) = u(b), \quad u'(a) = a'(b).
$$
If there exist smooth functions $\alpha(t)$, $\beta(t)$ and a continuous function $\varphi \colon \mathbb{R}^+ \to \mathbb{R}$ such that
\begin{align}
    \begin{split}
        & \alpha(t) \leqslant \beta(t),\\
        & \alpha''(t) \geqslant f(t, \alpha(t), \alpha'(t)), \quad \beta''(t) \leqslant f(t, \beta(t), \beta'(t)),\\
        & |f(t, u, v)| \leqslant \varphi(|v|), \quad \int\limits_0^\infty \frac{s\,ds}{\varphi(s)} = \infty,
    \end{split}
\end{align}
then the problem has a solution. The condition on the right hand side of the equation is called Nagumo condition. This condition allows one to obtain a priori estimates for solutions. In our results, similar condition has been obtained from the geometrical considerations. In a similar form this condition was initially proposed in \cite{bernstein1904} (also, see \cite{de2006two}).

Taking into account topological and geometrical considerations, we prove the existence of a periodic solution for more complex systems, defined on a manifold. In particular, we do not know any results on the existence of a periodic solution for a non-Hamiltonian spherical pendulum in an external field proved using Nagumo condition. The proofs for the pendulum with a moving support point can be found in \cite{srzednicki2019periodic}. This result can also be obtained from a result in \cite{bolotin2015calculus}. 

The above methods and ideas can be applied to various mechanical systems. Some of them are considered above. For instance, the result on the existence of a periodic solution for the point moving along a rotating curve can be applied to the so-called butterfly robot, a system that has been intensively studied recently \cite{cefalo2006energy,surov2015case,lynch1998roles}. The same methods can be applied to a rotating smooth surface with a mass point moving on it and one can expect the existence of a periodic solution provided the rotation is slow and smooth.

\section*{Acknowledgement}
This research was supported by a grant of the Russian Science Foundation (Project No. 19-71-30012).

\bibliographystyle{model1-num-names}
\bibliography{sample}







\end{document}